\documentclass[a4paper,11pt]{amsart}

\usepackage{epsfig}
\usepackage{amsthm,amsfonts}
\usepackage{amssymb,graphicx,color}
\usepackage[all]{xy}
\usepackage{verbatim}
\usepackage{hyperref}
\usepackage{enumitem}

\newtheorem{theorem}{Theorem}[section]
\newtheorem*{theorem*}{Theorem}
\newtheorem{lemma}[theorem]{Lemma}
\newtheorem{corollary}[theorem]{Corollary}
\newtheorem{proposition}[theorem]{Proposition}
\newtheorem{definition}[theorem]{Definition}

\newtheorem{remark}[theorem]{Remark}

\newcommand{\R}{\mathbb{R}}
\newcommand{\C}{\mathbb{C}}

\newtheorem{innercustomthm}{Theorem}
\newenvironment{customthm}[1]
  {\renewcommand\theinnercustomthm{#1}\innercustomthm}
  {\endinnercustomthm}

\begin{document}

\title[Bi-Lipschitz equivalent cones with different degrees]
{Bi-Lipschitz equivalent cones with different degrees}

\author[A. Fernandes]{Alexandre Fernandes}
\author[Z. Jelonek]{Zbigniew Jelonek}
\author[J. E. Sampaio]{Jos\'e Edson Sampaio}

\address[Alexandre Fernandes and Jos\'e Edson Sampaio]{    
              Departamento de Matem\'atica, Universidade Federal do Cear\'a,
	      Rua Campus do Pici, s/n, Bloco 914, Pici, 60440-900, 
	      Fortaleza-CE, Brazil. \newline  
              E-mail: {\tt alex@mat.ufc.br}\newline  
              E-mail: {\tt edsonsampaio@mat.ufc.br}
}
\address[Zbigniew Jelonek]{ Instytut Matematyczny, Polska Akademia Nauk, \'Sniadeckich 8, 00-656 Warszawa, Poland \& Departamento de Matem\'atica, Universidade Federal do Cear\'a,
	      Rua Campus do Pici, s/n, Bloco 914, Pici, 60440-900, 
	      Fortaleza-CE, Brazil. \newline  
              E-mail: {\tt najelone@cyf-kr.edu.pl}
}

\keywords{Links of cones, Bi-Lipschitz homeomorphism, Degree, Multiplicity mod 2}
\subjclass[2010]{14B05, 32S50, 58K30, 58K20}
\thanks{The first named author was partially supported by CNPq-Brazil grant 304221/2017-1. The second named author is partially supported by the grant of Narodowe Centrum Nauki number 2019/33/B/ST1/00755. The third named author was partially supported by CNPq-Brazil grant 303375/2025-6 and by the Serrapilheira Institute (grant number Serra -- R-2110-39576).
}
\begin{abstract}
We show that for every $k\ge 3$, there exist complex algebraic cones of dimension $k$ with isolated singularities, which are bi-Lipschitz and semi-algebraically equivalent but have different degrees. We also prove that homeomorphic projective hypersurfaces with dimension greater than 2 have the same degree. In the final part of the paper, we classify links of real cones with base $\mathbb{P}^1_{\mathbb{R}}\times \mathbb{P}^2_{\mathbb{R}}.$ As an application, we give an example of three four-dimensional real algebraic cones in $\mathbb{R}^8$ with isolated singularities, which are semi-algebraically and bi-Lipschitz equivalent but have non-homeomorphic bases. We also discover some new properties to study the links of real algebraic varieties. In particular, we show that for cones over $\mathbb{P}_{\R}^1\times \mathbb{P}_{\R}^2$ and for cones over $\mathbb{P}_{\R}^k$, the multiplicity mod 2 is a metric invariant.
Moreover, we give examples of real algebraic manifolds that are not diffeomorphic to projective manifolds of odd degree.
\end{abstract}

\maketitle
\tableofcontents

\section{Introduction}

In 1971, O. Zariski \cite{Zariski:1971} proposed many questions, and the most well-known among them is the following.
\begin{enumerate}[leftmargin=0pt]
\item[]{\bf Question A.} Let $f,g\colon(\mathbb{C}^n,0)\to (\mathbb{C},0)$ be two complex analytic functions. If there is a homeomorphism $\varphi\colon(\mathbb{C}^n,V(f),0)\to (\mathbb{C}^n,V(g),0)$, is it true that the multiplicities $m(V(f),0)$ and $m(V(g),0)$ are equal?
\end{enumerate}

This is still an open problem. The stated version of Question A is Zariski's famous Multiplicity Conjecture. Recently, Zariski's Multiplicity Conjecture for families with isolated singularities was confirmed by Fern\'andez de Bobadilla and Pe\l ka \cite{BobadillaP:2022}.

Recently, there have also been some contributions to Zariski's Multiplicity Conjecture from the Lipschitz point of view. 
For instance, in \cite{BobadillaFS:2018} the following conjecture was proposed
(see the definition of bi-Lipschitz homeomorphism in Definition \ref{defi:bi_lip}):
\vspace{3mm}

\begin{enumerate}[leftmargin=0pt]
\item[]{\bf General Metric Zariski Multiplicity Conjecture.} \label{conj_local}
{\it Let $X\subset \C^n$ and $Y\subset \C^m$ be two complex analytic sets with $\dim X=\dim Y=d$. If there is a bi-Lipschitz homeomorphism $\varphi\colon(X,0)\to (Y,0)$, then the multiplicities $m(X,0)$ and $m(Y,0)$ are equal.}
\end{enumerate}

This conjecture has a real counterpart:

\begin{enumerate}[leftmargin=0pt]
\item[]{\bf Conjecture A$_{\R}$(Lip).} \label{conj_val}
{\it Let $X\subset \R^n$ and $Y\subset \R^m$ be two real analytic sets with $\dim X=\dim Y=d$. If there is a bi-Lipschitz homeomorphism $\varphi\colon(X,0)\to (Y,0)$, then   $m(X,0)=m(Y,0)\mod 2.$}
\end{enumerate}

This conjecture appears in \cite{Sampaio:2022b}, but it was already proposed as a question in the case of $m=n$ and $d=n-1$ by Valette (see the introduction of \cite{val}).

Already in \cite{BobadillaFS:2018}, the authors proved that the General Metric Zariski Multiplicity Conjecture has a positive answer for $d=2$.
The positive answer for $d=1$ was already known, since Neumann and Pichon  \cite{N-P}, with previous contributions from Pham and Teissier  \cite{P-T} and Fernandes  \cite{F}, proved that the Puiseux pairs of plane curves are invariant under bi-Lipschitz homeomorphisms, and as a consequence, the multiplicity of complex analytic curves of any codimension is invariant under bi-Lipschitz homeomorphisms. In order to find other partial results concerning the General Metric Conjecture see, e.g., \cite{BirbrairFLS:2016}, \cite{Comte:1998}, \cite{ComteMT:2002}, \cite{FernandesS:2016}, \cite{Jelonek:2021}, \cite{Sampaio:2016}, \cite{Sampaio:2019}, \cite{Sampaio:2020b} and \cite{Sampaio:2022}. However, in dimension  three, Birbrair, Fernandes, Sampaio, and Verbitsky   \cite{bfsv} have presented examples of complex algebraic cones over the smooth quadric, which were bi-Lipschitz homeomorphic but had different multiplicities at the origin. To prove their result, Birbrair, Fernandes, Sampaio, and Verbitsky used the theory of Smale-Barden manifolds. See \cite{FernandesS:2023} and \cite{FernandesS:2025} for an overview of the invariance of the multiplicity under bi-Lipschitz homeomorphisms.

On the other hand, the Conjecture A$_{\R}$(Lip) is still almost completely open. We have proved it only for sub-analytic and arc-analytic bi-Lipschitz mapping (see \cite{fjs}). We go back to this problem in the last section.

The first aim of this paper is to generalize the result from \cite{bfsv}. We show that for every $k\ge 3$, there exist complex algebraic cones of dimension $k$ with isolated singularities, which are bi-Lipschitz and semi-algebraically equivalent but have different degrees (see Theorem \ref{thm:complex_cones}). Our proof is completely different from the previous one, and it is based on the Steenrod Theorem on sphere bundles.

Our result is related to the converse of Question B in \cite{Zariski:1971}, which is also related to Question A (see \cite{Sampaio:2020a}). 

Now we state the Zariski Question B \cite{Zariski:1971} (in a somewhat simplified version). 
Let  $E_{0}(X)$ be the projectivized tangent cone to $X$ at $0.$ 

\begin{enumerate}[leftmargin=0pt]
\item[]{\bf Question B.} {\it Let $f,g\colon(\mathbb{C}^n,0)\to (\mathbb{C},0)$ be two complex analytic functions. If there is a homeomorphism $\varphi\colon(\mathbb{C}^n,V(f),0)\to (\mathbb{C}^n,V(g),0)$, is there a homeomorphism $h\colon E_{0}(V(f))\to E_{0}(V(g))$?  }
\end{enumerate}

This problem has a negative answer, as shown by Fern\'andez de Bobadilla \cite{Bobadilla:2005}. However, the Metric Question B is still open:

\begin{enumerate}[leftmargin=0pt]
\item[]{\bf Metric Question B.} {\it Let $f,g\colon(\mathbb{C}^n,0)\to (\mathbb{C},0)$ be two complex analytic functions. If there is a bi-Lipschitz homeomorphism $\varphi\colon(\mathbb{C}^n,V(f),0)\to (\mathbb{C}^n, $ $V(g),0)$, is there a homeomorphism $h\colon E_{0}(V(f))\to E_{0} $ $(V(g))$?  }
\end{enumerate}

In \cite[Conclusion, p. 129]{kol2} Koll\'ar proved that if $X\subset \mathbb{P}_{\C}^{n+1}$ is a smooth projective hypersurface of dimension greater than one, then the degree of $X$ is determined by the underlying topological space of $X$. In \cite[Theorem 1]{BarthelD:1994} Barthel and Dimca proved that in the case of projective hypersurfaces (possibly with singularities) of dimension greater than one, degree one is a topological invariant.
Here, we generalize these results. More precisely, we prove the following:

\begin{customthm}{\ref*{thm:gen_kollar1}}
Let $V\subset \mathbb{P}_{\C}^{m+1}, V'\subset \mathbb{P}_{\C}^{l+1} $ be two projective varieties of dimension $n>2$, which are  set theoretic complete intersections. If they are homeomorphic, then $\deg V=\deg V'.$ 
\end{customthm}

As a consequence of this result, we show in Corollary \ref{cor:question_b_imples_a} that a positive answer to Metric Question B implies a positive answer to the following (see also \cite[Question 3.6.4]{Bobadilla:2022}):
\begin{enumerate}[leftmargin=0pt]
\item[]{\bf Metric Question A.} {\it Let $f,g\colon(\mathbb{C}^n,0)\to (\mathbb{C},0)$ be two complex analytic functions. If there is a bi-Lipschitz homeomorphism $\varphi\colon(\mathbb{C}^n,V(f),0)\to (\mathbb{C}^n,V(g),0)$, is it true that $m(V(f),0)=m(V(g),0)$?}
\end{enumerate}

The second question that we have in mind has the following more general statement:
\begin{enumerate}[leftmargin=0pt]
\item[]{\bf General Metric Question B.} {\it Let $X\subset \C^n$ and $Y\subset \C^m$ be two complex analytic sets with $\dim X=\dim Y=d$. If $(X,0)$ and $(Y,0)$ are bi-Lipschitz homeomorphic, is there a homeomorphism $h\colon E_{0}(X)\to E_{0}(Y)$? }
\end{enumerate}

Note that the version of this question, where it is asked whether bi-Lipschitz homeomorphic complex analytic sets have bi-Lipschitz homeomorphic tangent cones, was positively solved in \cite{Sampaio:2016} (see also \cite{SampaioS:2022}).

In the final part of this paper, we classify links of real cones with base $\mathbb{P}_{\R}^1\times \mathbb{P}_{\R}^2$, see Theorem
\ref{thm:char_P1xP2}. As an application,  we give an example of three four-dimensional real algebraic cones in $\R^8$ with isolated singularity that are semi-algebraically and bi-Lipschitz equivalent but have non-homeomorphic bases, see Theorem \ref{main_thm}.
 Consequently, the real version of the General Metric Question B has a negative answer.

Finally, we give examples of compact manifolds that are not diffeomorphic to real projective manifolds of odd degree, see Theorem \ref{ost} (1) and (3). In particular, we show that for cones over $\mathbb{P}_{\R}^1\times \mathbb{P}_{\R}^2$ and for cones over $\mathbb{P}_{\R}^k$ the Conjecture A$_{\R}$(Lip) has a positive answer, see Corollary \ref{conj}.
 
\vspace{5mm}

{\bf Acknowledgements.} The authors are grateful to professor W\l odzimierz Jelonek from Krak\'ow and to professor Adam Parusi\'nski from Nice for helpful discussions. The authors also would like to thank the anonymous referees for their valuable suggestions and comments that helped to improve the text.

\section{Preliminaries}\label{section:preliminaries}

\begin{definition}\label{defi:bi_lip}
Let $X\subset \R^n$ and $Y\subset \R^m$ be two sets and let $h\colon X\to Y$.
\begin{itemize}
 \item We say that $h$ is {\bf Lipschitz} if there exists a positive constant $C$ such that
$$\|h(x)-h(y)\|\leq C\|x-y\|, \quad \forall x, y\in X.
$$ 
\item We say that $h$ is {\bf bi-Lipschitz}  if $h$ is a homeomorphism, it is Lipschitz and its inverse  is also Lipschitz. In this case, we say that $X$ and $Y$ are {\bf bi-Lipschitz equivalent} or {\bf bi-Lipschitz homeomorphic}. When $n=m$ and $h$ is the restriction of a bi-Lipschitz homeomorphism $H\colon \R^n\to \R^n$, we say that $X$ and $Y$ are {\bf ambient bi-Lipschitz equivalent}.

\end{itemize}
\end{definition}

\begin{definition}\label{cone}
Let $k=\R$ or $\C$. Let $X\subset \mathbb{P}_k^n$ be an algebraic variety. By the affine $k$-cone $C_k(X)$, we mean the homogeneous affine variety in $k^{n+1}$ defined by the same homogeneous equations as $X$. By the algebraic cone $\overline{C_k(X)}\subset \mathbb{P}_k^{n+1}$ with base $X$, we mean the projective closure in $\mathbb{P}^{n+1}_k$ of $C_k(X)$.
Geometrically:
$$\overline{C_k(X)}=\bigcup_{x\in X} \overline{O,x},$$ where $O$ is the center of coordinates in $ k^{n+1}\subset \mathbb {P}^{n+1}_k$, and $\overline{O,x}$ means the projective line which goes through $O$ and $x.$ 
By the link of $C(X)$ we mean the set $L=\{ x\in C(X): ||x||=1\}.$
\end{definition}

\begin{remark}
{\rm Since the cone is homogenous, our definition of the link coincides with the standard one. }  
\end{remark}

\begin{definition}\label{hopf}
Let $k=\R$ or $k=\C$. Let us consider the affine space $\mathbb{A}^n_k\subset \mathbb{P}_k^n$. Let $S$ be the unit sphere in $\mathbb{A}^n_k.$ By the Hopf fibration, we mean the mapping $\pi: S \ni x\mapsto [x]\in \mathbb{P}^{n-1}_k$. Note that if $k=\R$, then the Hopf fibration is a covering of degree two, if $k=\C$ then it is a circle fibration. If $L$ is a link of a cone $C_k(X)$, then by the Hopf fibration $L\to X$, we mean the mapping $\pi_{|L}.$   \end{definition}

\begin{remark}
{\rm It is easy to see that the Hopf fibration is a real analytic mapping.}    
\end{remark}

The following result is well-known, but for the convenience of the reader, we give the proof.

\begin{proposition}\label{alex}
Let $k=\R$ or $\C$. Let $C_k(X)$ and $C_k(Y)$ be affine $k$-cones in $k^N.$ Assume that their links are bi-Lipschitz (semi-algebraically) equivalent. Then they are bi-Lipschitz (semi-algebraically) equivalent. Moreover, if $\dim C_k(X)=\dim C_k(Y)=d$, $2d+2\le N$, and $C_k(X)$ is semi-algebraically bi-Lipschitz equivalent to $C_k(Y)$, then they are ambient semi-algebraically bi-Lipschitz equivalent.
\end{proposition}

\begin{proof}
Let $X',Y'$ denote the links of $C(X),C(Y)$ respectively. We have  $X',Y'\subset \mathbb{S}^{N-1}$, where $\mathbb{S}^{N-1}$ denotes the unit Euclidean sphere in $\mathbb{R}^{N}$. In this case, $$C(X) =\{ t\cdot x \ : \ x\in X' \ \mbox{and} \ t\geq 0\} \ \mbox{and} \  C(Y) =\{ t\cdot y \ : \ y\in Y' \ \mbox{and} \ t\geq 0\}.$$ By assumption, there is a bi-Lipschitz homeomorphism $f\colon X'\rightarrow Y'$, i.e. there is $\lambda \geq 1$ such that  $$\frac{1}{\lambda} \| x_1-x_2 \| \leq \| f(x_1)-f(x_2) \| \leq \lambda \| x_1 - x_2 \| \ \forall \ x_1,x_2\in X'.$$ Let us define $F\colon C(X)\rightarrow C(Y)$ by $F(t\cdot x) = t\cdot f(x)$ for all $ x\in X'$ and $t\geq 0$. We claim that $F$ is a bi-Lipschitz map. In fact, given $t\cdot x_1, s\cdot x_2\in C(X)$ (one may suppose that $t\le s$), we have
\begin{eqnarray*}
	\| F(t\cdot x_1) - F(s\cdot x_2) \| &\leq & \| F(t\cdot x_1) - F(t\cdot x_2)\| + \| F(t\cdot x_2) - F(s\cdot x_2)\| \\
	 &=& t \| f(x_1) - f(x_2) \| + |t-s|\| f(x_2) \| \\
	 &\leq& \lambda t \| x_1 - x_2 \| + |t-s| \|x_2\| \\
	 &=& \lambda \| t\cdot x_1 - t\cdot x_2 \| + \| t\cdot x_2 -s\cdot x_2\|
     \end{eqnarray*}

     Note that  $\|x_1\|=\|x_2\|= 1$ and $0\le t\le s$. Hence $ \|t\cdot x_1 - s\cdot x_2\|^2= \|t\cdot x_1-t\cdot x_2-(s-t)\cdot x_2\|^2=\|t\cdot x_1 - t\cdot x_2||^2+2(s-t)t(1-x_1\circ x_2)+|s-t|^2\geq\|t\cdot x_1 - t\cdot x_2\|^2$ and finally 
     $||t\cdot x_1 - s\cdot x_2||\geq ||t\cdot x_1 - t\cdot x_2||$ (the symbol $\circ$ means here the scalar product). In a similar way $\| t\cdot x_1 -s\cdot x_2\|\geq\| t\cdot x_2 -s\cdot x_2\|$ hence:

     \begin{eqnarray*}
     \| F(t\cdot x_1) - F(s\cdot x_2) \| 
	 &\leq & \lambda \| t\cdot x_1 - s\cdot x_2 \| + \| t\cdot x_1 -s\cdot x_2\|\\
  &=& (\lambda + 1) \| t\cdot x_1 - s\cdot x_2 \|.
\end{eqnarray*}

By the symmetry  the inverse map of $F$ is also a  Lipschitz map. This means that $F$ is a bi-Lipschitz map.
Moreover, if $f$ is additionally semi-algebraic, we see that $F$ is also semi-algebraic by construction, and  similarly its inverse.
The last statement follows directly from \cite{bfj}.
\end{proof}

\begin{remark}
 {\rm Two semi-algebraic cones are semi-algebraically bi-Lipschitz equivalent   if and only if they have semi-algebraically bi-Lipschitz equivalent links (see \cite{val?}).}
 \end{remark}

\begin{definition}
An $k$-dimensional  subanalytic set $C\subset\R^n$ is said to be an {\bf Euler cycle} if it is a closed set and if, for some  triangulation of $C$ (and hence for any that refines it), the number of $k$-dimensional simplexes containing a given $(k-1)$-dimensional simplex  is even.
\end{definition}

In particular, if $C$ is compact  and it is an Euler cycle, then it is $\Bbb Z/2$   homological cycle. Indeed, if $T$ is a suitable triangulation, then $C$  can be treated as a sum of $k$-dimensional simplexes from $T.$

\begin{definition}
Let $a:\R^n\to\R^n$ be the antipodal mapping, i.e., $a(x)=-x.$ We say that a set $C \subset \mathbb{R}^n$ is {\bf $a$-invariant} if it is preserved by the antipodal mapping.
\end{definition}

\begin{remark}
{\rm Let $C$ be a $a$-invariant Euler cycle, with $a$-invariant triangulation $T$. Let $\overline{C}$ be the projective closure   of  $C.$ Then there exists an $a$-invariant triangulation of $\overline{C}$ which on $C$ is a  refinement of  the triangulation
$T$, such that $\overline{C}$ with this triangulation is a $\mathbb{Z}/(2)$ homological cycle.

\vspace{5mm}

Indeed, take a sufficiently large sphere $S$ with the center at $0$, which is transversal to $C.$ Let $T'$ be an $a$-invariant  refinement of the triangulation $T$, which is compatible with $C\cap S.$ Then $C$ with this triangulation is also an Euler cycle. Take a projective closure 
$\overline{C}$ of this  cycle, and triangulation of $\overline{C}$ which is compatible with $T'$. Let $S$ be a $n-1$ dimensional stratum at infinity.
If $A\subset \overline{C}$ is a $n$ dimensional simplex that contains $S$, then $a(A)$ is also such a simplex. But by the construction, we have $a(A)\not=A.$}
\end{remark}

The following theorem has been proved in \cite[Theorem 3.2]{fjs}: 

\begin{theorem}\label{Ztwo_cycle}
Let $X\subset \R^n,Y\subset \R^m$ be real algebraic sets, and let $h\colon X\to Y$ be a subanalytic and bi-Lipschitz homeomorphism. Assume that the projective closure of the graph of $h$ is a $\mathbb{Z}/(2)$ homological cycle. Then
${\rm deg} (X)= {\rm deg } (Y) \ mod \ 2.$
\end{theorem}

\begin{corollary}\label{bilipszic}
 Let $X\subset \R^n,Y\subset \R^m$ be real algebraic sets, and let $h\colon X\to Y$ be a subanalytic and bi-Lipschitz homeomorphism. Assume that the graph of $h$ is an $a$-invariant Euler   cycle in $\R^n\times \R^m.$ Then
${\rm deg} (X)= {\rm deg } (Y) \ mod \ 2.$   
\end{corollary}

\section{Complex cones}
In this Section, we generalise the main result in \cite{bfsv}.

Let $C_k$ denote the Veronese embedding of degree $k$ of $\mathbb{P}_{\C}^1$ into $\mathbb{P}_{\C}^k$ given by   $\psi([z_0,z_1])=[z_0^k:z_0^{k-1}z_1:....,z_0z_1^{k-1}:z_1^k]$ (see \cite[Example 8.4.3 b)]{Fulton} or \cite[Example 18.13]{Harris:1992}). Let $n\ge 2$ and consider the varieties $X_{k,n}=\phi( C_k\times \mathbb{P}_{\C}^{n-1})$, where $\phi:\mathbb{P}_{\C}^k\times \mathbb{P}_{\C}^1\rightarrow \mathbb{P}_{\C}^{n_k},  \phi([z_0:z_1:...:z_k],[w_0,w_1])=[z_0w_0:z_0w_1:z_1w_0:z_1w_1,...:z_kw_0:z_kw_1]$ is the Segre embedding (see \cite[Example 8.4.3 a)]{Fulton} or \cite[Example 18.15]{Harris:1992}).

\begin{theorem}\label{thm:complex_cones}
For each fixed $n$, all varieties $X_{k,n}$ have different degrees, and among the cones $C_{\C}(X_{k,n})$, there are infinitely many cones that are bi-Lipschitz and semi-algebraically equivalent.
\end{theorem}

\begin{proof}

Note that $C_k$ is $k\mathbb{P}_{\C}^1$ as a cycle. Hence $C_k\times \mathbb{P}_{\C}^{n-1}\sim k \mathbb{P}_{\C}^1\times \mathbb{P}_{\C}^{n-1}$, where $\Gamma \sim \Lambda$ means that $\Gamma$ and $ \Lambda$ are homologous cycles. Since after the Segre embedding $\deg\mathbb{P}_{\C}^1\times \mathbb{P}_{\C}^{n-1}=n$, we have deg $X_{k,n}=kn$  (see again \cite[Example 8.4.3]{Fulton} or \cite[Examples 18.13 and 18.15]{Harris:1992}).

By construction, $X_{k,n}$ is the  union of projective $(n-1)$-planes 
$X_{k,n}=\bigcup_{a\in C_k} \phi(\{a\} \times \mathbb{P}_{\C}^{n-1})$. This means that $\overline{C_{\C}(X_{k,n})}$ is the  union of $n$-planes, which have the $(n-1)$-plane
$\phi(\{a\}\times\mathbb{P}_{\C}^{n-1})$ at infinity and go through the point $O=(0,...,0)$ (see Definition \ref{cone}). Thus the link $L_{k,n}$ of this cone is a union of $(2n-1)$-spheres. In fact, using the Ehresmann Theorem, it is easy to observe that  these links are fibre  bundles over $C_k\cong \mathbb{S}^2$ with the projection being the composition of the projection $p: \mathbb{P}_{\C}^{n_k+1} \setminus O\to \mathbb{P}_{\C}^{n_k}$ restricted to $L_{k,n}$ and the projection $q:  C_k\times\mathbb{P}_{\C}^{n-1}\to C_k.$ Since all fibers are great spheres on a fixed unit sphere, by the Steenrod Theorem  \cite[Th.2, p. 298]{steenrod}, we see that the structure group of this fibre bundle can be reduced to the full orthogonal group, hence this fibration is a sphere bundle in the sense of Steenrod.  By the Steenrod Theorem \cite[6.III, p. 300]{steenrod}) topologically there are only two such sphere bundles. On the other hand, it follows from \cite[Classification Theorem, p. 155]{KirbyS:1977} that on a compact manifold of dimension different from four  there is only a finite number of differential structures.  This means that all manifolds $L_{k,n}$ , $k=1,2,...,$ can have only a finite number of different differential structures. By the Dirichlet box principle, among all $X_{k,n}$, there is an infinite family $\mathcal{S}$  whose  members are diffeomorphic to each  other.

By \cite[Corollary 11]{kol}, all links from the family $\mathcal{S}$ are Nash diffeomorphic. In particular, they are bi-Lipschitz and semi-algebraically equivalent. By Proposition \ref{alex}, we see that all cones $C_{\C}(X), X\in \mathcal{S}$, are bi-Lipschitz and semi-algebraically equivalent. But all members of the family $\{C_{\C}(X): X\in \mathcal{S}\}$ have different degrees.
\end{proof}

\begin{corollary}
For every $n\ge 3$ there exist infinitely many analytic $n$ dimensional germs $V_i\subset (\C^{2n},0), i=1,2,....,$ with isolated singularities, which are bi-Lipschitz, sub-analytically equivalent, but have pairwise different multiplicities at $0.$
\end{corollary}

\begin{proof}
Using a generic projection, we can assume that all $X_{k,n}$ are in $\mathbb{\C}^{2n},$
see  Theorem 5.5 in \cite{bfj} (note that $X_{k,n}$ is a cone, hence we can use a local method).
\end{proof}

\section{Topological invariance of the degree of a projective set-theoretic complete intersection}  

The following result is a rather simple consequence of the Lefschetz theorem. However, despite this, it was not known before, see e.g. the book \cite{dimca}, where it is stated only in the  linear and  smooth case (see Corollary 2.12 and Exercise 3.35 in \cite{dimca}). Since it is fundamental for us, we have decided to give a full detailed proof here.

\begin{theorem}\label{thm:gen_kollar1}
Let $V\subset \mathbb{P}_{\C}^{m+1}, V'\subset \mathbb{P}_{\C}^{l+1} $ be two projective varieties of dimension $n>2$, which are  set theoretic complete intersections. If they are homeomorphic, then $\deg V=\deg V'.$ 
\end{theorem}

\begin{proof}
Our proof was inspired by that in \cite{BarthelD:1994}.
Let $V_1,\ldots,V_r$ (resp. $V_1',\ldots,V_s'$) be the irreducible components of $V$ (resp.  $V'$).
Let $\phi\colon V\to V'$ be a homeomorphism. By \cite[Lemma A.8]{Gau-Lipman:1983}, $\phi(V_j)$ is an irreducible component of $V'$ for all $j=1,...,r$. Then $r=s$, and by reordering the indices if necessary, we can assume that $\phi(V_j)=V'_j$ for all $j=1,...,r$.
Note that  $\deg V =\deg V_1 +...+\deg V_r$ and $\deg V' =\deg V'_1+...+\deg V'_r.$

Let us recall that the cohomology ring of   $\mathbb{P}_{\C}^{m+1}$ is isomorphic to $\mathbb{Z}[x]/(x^{m+2})$ (see e.g. \cite{greenberg}) and it is generated by the cohomology class  $\alpha\in H^2(\mathbb{P}_{\C}^{m+1},\mathbb{Z})$ dual to the class of hyperplane.
Let $\iota: V\to \mathbb{P}_{\C}^{m+1}$ be the inclusion. By Lefschetz theorem (see \cite{dimca}, p. 143) and our assumption $n>2$, the mapping $\iota^*:  H^2(\mathbb{P}_{\C}^{m+1},\mathbb{Z})\to H^2(V,\mathbb{Z})$ is an isomorphism. In particular, the element
$\alpha_V=\iota^*(\alpha)$ is a generator of  $H^2(V,\mathbb{Z}).$   

By the Mayer-Vietoris exact sequence, we have $$H_{2n}(V,\mathbb{Z})=\bigoplus_{i+1}^r H_{2n}(V_i, \mathbb{Z})=\mathbb{Z}^r$$ (it is generated by the fundamental classes $[V_i]$). Moreover, by the relative long exact cohomology sequence,
$H^{2n} (V,\Bbb Z) = H^{2n} (V,Sing(V), \Bbb Z).$
By Lefschetz duality, we have $$H^{2n} (V,Sing(V), \Bbb Z)=H_0(V \setminus Sing(V), \Bbb Z).$$
We have $H_0(V \setminus Sing(V), \Bbb Z)=\Bbb Z^r.$ Consequently,  $H^{2n} (V,\Bbb Z) = \Bbb Z^r$ (see  \cite{max}, 4.1).

Since we have a canonical epimorphism $H^{2n}(V,\mathbb{Z}) \to H_{2n}(V,\mathbb{Z})^*$, we see that these spaces are isomorphic (see \cite[23.7]{greenberg}).
In fact, the restriction mapping $H^{2n}(\mathbb{P}_{\C}^{m+1},\mathbb{Z})\to H^{2n}(V, \mathbb{Z})$ is dual to the inclusion mapping  $\iota: H_{2n}(V, \mathbb{Z})\to H_{2n}(\mathbb{P}_{\C}^{m+1},\mathbb{Z})$ (see \cite[23.11]{greenberg}). The mapping $\iota$ sends a cycle $\sum^r_{i=1}\beta_i[V_i]\in H_{2n}(V)$ to the cycle  $\sum ^r_{i=1} \beta_i[V_i]\in H_{2n}(\mathbb{P}_{\C}^{m+1},\mathbb{Z})$.

This means that $\iota^*(\alpha^n)(\sum_{i=1}^r \beta_i[V_i])=    \sum_{i=1}^r \beta_i\alpha^n([V_i]).$ Since $V_i$ as a cycle is equal to $(\deg V_i)\tilde{\alpha}^{m+1-n}$, where $\tilde{\alpha}$ is the homology class of a hyperplane, we have $\iota^*(\alpha^n)(\sum_{i=1}^r \beta_i[V_i])=\sum^r_{i=1} \beta_i \deg V_i.$
In other words
$$\iota^*(\alpha^n)=\alpha_V^n=\sum^r_{i=1} \deg V_i[V_i]^*,$$
\noindent where $[V_i]^*$ is the (dual) fundamental class, i.e., $[V_i]^*([V_j])=\delta_{ij}$.

Now let $\alpha_{V'}$ be a generator of $H^2(V',\mathbb{Z})$ constructed in an analogous way to $\alpha_V$. Hence, by symmetry we have $\alpha_{V'}^n=\sum^r_{i=1} \deg V'_i[V'_i]^*$. Let $\phi\colon V\to V'$ be a homeomorphism. Note that $\phi^*([V'_i]^*)([V_j])=[V'_i]^*(\phi_*([V_j])=[V'_i]^*(\pm [V'_j])=\pm \delta_{ij}.$
Thus $\phi^*([V'_i]^*)=\pm [V_i]^*$. Moreover,  $\phi^*(\alpha_{V'})=\pm \alpha_V.$ 
Consequently
$$
 \sum^r_{i=1}\deg V_i' \cdot (\pm [V_i]^*)=\phi^*(\deg V_i' \cdot [V_i']^*)=\phi^*(\alpha_{V'}^n)=\pm \alpha_V^n = \pm \sum^r_{i=1}\deg V_i \cdot [V_i]^*.
$$ 
Hence $\deg V_i=\deg V_i'$ for $i=1,...,r.$
\end{proof}

\begin{corollary}\label{cor:gen_kollar}
Let $V,V'\subset \mathbb{P}_{\C}^{n+1}$ be two projective hypersurfaces. Assume $n>2.$ If $V$ is homeomorphic to $V'$, then $\deg V=\deg V'.$
\end{corollary}

\begin{remark}
    {\rm If varieties $V$ and $V'$ are additionally irreducible, then the proof of Theorem \ref{thm:gen_kollar1} simplifies and immediately shows that, in this case, the degree is  also a homotopy invariant.}
\end{remark}




\begin{remark}
{\rm (See \cite[Chapter 5, Example 3.34]{dimca}) Theorem \ref{thm:gen_kollar1} is not true if $n\le 2.$ Indeed, it is easy for $n=1.$ For $n=2$, it follows from 
\cite{E} and \cite{L-W}. In fact, the complete intersections  $X=V(10,7,7,6,3,3)$ in $\mathbb{P}_{\C}^8$ and  $Y=V(9,5,3,3,3,3,3,2,2)$
in $\mathbb{P}_{\C}^{11}$ are homeomorphic, but they have different degrees (we  use here the notation from \cite{dimca}). }
\end{remark}

\begin{corollary}\label{wniosek}
 Let $f,g\colon(\mathbb{C}^n,0)\to (\mathbb{C},0)$ be two complex analytic functions with $n>4$. Assume that there is a bi-Lipschitz homeomorphism $\varphi\colon(\mathbb{C}^n,V(f),0)\to (\mathbb{C}^n,V(g),0)$. If there is a homeomorphism $h\colon E_{0}(V(f))\to E_{0}(V(g))$, then $m(V(f),0)=m(V(g),0)$.
\end{corollary}
\begin{proof}
By \cite[Theorem 2.1]{FernandesS:2016}, we may assume that $f$ and $g$ are irreducible homogeneous polynomials. In this case, $m(V(f),0)=\deg E_{0}(V(f))$ and $m(V(g),0)=\deg E_{0}(V(g))$. Since there is a homeomorphism $h\colon E_{0}(V(f))\to E_{0}(V(g))$, it follows from Corollary \ref{cor:gen_kollar} that $m(V(f),0)=m(V(g),0)$.
\end{proof}

Thus, we obtain the following:

\begin{corollary}\label{cor:question_b_imples_a}
 If Metric Question B has a positive answer, then Metric Question A has a positive answer as well.
\end{corollary}

\begin{proof}
Let $V,V'$ be bi-Lipschitz equivalent germs of analytic sets. We can assume that $V,V'$ are bi-Lipschitz equivalent cones. If $n>4$ then the result follows from \ref{wniosek}. In the general case, let us consider bi-Lipschitz equivalent cones $W=V\times \C^4\subset\C^{n+4}$
and $W'=V'\times\C^4\subset\C^{n+4}.$ Then $\deg W=\deg W'$, hence $\deg V=\deg V'$.
\end{proof}

\section{Real cones}
In this section, we consider  real algebraic varieties. We prove the following:

\begin{theorem}\label{thm:char_P1xP2}
Let  $\iota : \mathbb{P}_{\R}^1\times \mathbb{P}_{\R}^2\to \mathbb{P}_{\R}^n$ be an algebraic embedding. Let $X=\iota( \mathbb{P}_{\R}^1\times \mathbb{P}_{\R}^2).$ If $\deg X$ is odd, then the link of the cone $C_{\R}(X)$ is diffeomorphic to the twisted product $\tilde{\mathbb{S}^1\times \mathbb{S}^2}$. If $\deg X$ is even, and the link of $C_{\R}(X)$ is connected, then it is diffeomorphic either to $\mathbb{P}_{\R}^1\times\mathbb{P}_{\R}^2$ or to $\mathbb{P}_{\R}^1\times \mathbb{S}^2$, and both cases are possible. 
\end{theorem}

\begin{proof}
Denote by $A_k\subset \mathbb{P}_{\R}^k$ and $B_l\subset \mathbb{P}_{\R}^{\genfrac(){0pt}{2}{l+2}{2}-1}$ the Veronese embedding of $\mathbb{P}_{\R}^1$ and $\mathbb{P}_{\R}^2$ of degree $k$ and $l$, respectively  (see \cite[Example 8.4.3 b)]{Fulton} or \cite[Example 18.13]{Harris:1992}).
Now let  $\phi: \mathbb{P}_{\R}^k\times \mathbb{P}_{\R}^{\genfrac(){0pt}{2}{l+2}{2}-1}  \to \mathbb{P}_{\R}^{N(k,l)}$ be the Segre embedding  (see \cite[Example 8.4.3 a)]{Fulton} or \cite[Example 18.15]{Harris:1992}), and denote by $W_{k,l}$ the image $\phi(A_k \times B_l)$. As in the previous section, we have that deg $W_{k,l} = 3kl.$
Indeed, note that $A_k$ is $k\mathbb{P}_{\R}^1$ and $B_l$ is $l\mathbb{P}_{\R}^2$ as cycles. Hence $A_k\times B_l\sim (k \mathbb{P}_{\R}^1)\times (l\mathbb{P}_{\R}^{2})\sim kl \mathbb{P}_{\R}^1\times \mathbb{P}_{\R}^{2}$, where $\Gamma \sim \Lambda$ means again that $\Gamma$ and $ \Lambda$ are homologous cycles.  Since the Segre embedding of $\mathbb{P}_{\R}^1\times \mathbb{P}_{\R}^{2}$ has degree $3$ (see \cite[Example 8.4.3 a)]{Fulton}), we have that $\deg\phi(\mathbb{P}_{\R}^1\times \mathbb{P}_{\R}^{2})=3$. 
Then
$W_{k,l}=\phi(A_k\times B_l)\sim kl \phi(\mathbb{P}_{\R}^1\times \mathbb{P}_{\R}^{2})$.
Therefore, deg $W_{k,l} = 3kl.$

Let 
$X_{k,l}=C_{\R}(W_{k,l})$ be the cone with base $W_{k,l}.$ Additionally, denote by $L_{k,l}$ the link of this cone.

By the constructions, every base $W_{k,1}$ is the union of planes 
$$W_{k,1}=\bigcup_{a\in A_k } \phi(\{a\}  \times \mathbb{P}_{\R}^2).$$ 
This means that $\overline{X_{k,1}}$ is the union of $3$-planes which have the plane
 $\phi(\{a\}  \times \mathbb{P}_{\R}^2)$ at infinity and go through the point $O=(0,...,0).$ 
 Similarly  $\overline{X_{1,l}}$ is the union of planes which have the line
$\phi_1(\mathbb{P}_{\R}^1\times \{a\})$ at infinity and go through the point $O=(0,...,0).$

Thus the link  of $X_{k,1}$ is a union of spheres and  the link of $X_{1,l}$ is a union of  circles. In fact, it is easy to observe that the former link is a  sphere   bundle over $\mathbb{P}_{\R}^1$ whose projection is a composition of the projection 
$ \mathbb{P}_{\R}^8\setminus \{O\}\to \mathbb{P}_{\R}^{7}$  and the projection $: \mathbb{P}_{\R}^1\times \mathbb{P}_{\R}^2\to \mathbb{P}_{\R}^1$ . Similarly, the link of $X_{1,k}$ is a circle bundle over $\mathbb{P}_{\R}^2$  whose projection is a composition of the projection 
$  \mathbb{P}_{\R}^8\setminus \{O\}\to \mathbb{P}_{\R}^{7}$  and the projection  $ \mathbb{P}_{\R}^1\times \mathbb{P}_{\R}^2\to \mathbb{P}_{\R}^2$. In particular, both links are connected. Note that the  link $L_{1,1}$ has the structure of a circle bundle over $\mathbb{P}_{\R}^2$ and the structure of a sphere bundle 
over $\mathbb{P}_{\R}^1.$  We have:

\begin{lemma}\label{lemma}
If the link over a cone with base $\mathbb{P}_{\R}^1\times \mathbb{P}_{\R}^2$ is connected, then it is diffeomorphic either to $\mathbb{P}_{\R}^1\times  \mathbb{P}_{\R}^2$,  or to $\mathbb{P}_{\R}^1\times \mathbb{S}^2$, or to the twisted product $\tilde{\mathbb{S}^1\times \mathbb{S}^2}=\mathbb{S}^1\times \mathbb{S}^2/G$, where $G$ is the group generated by the involution $g: \mathbb{S}^1\times \mathbb{S}^2\ni (x,p)\mapsto (-x,-p)\in \mathbb{S}^1\times \mathbb{S}^2.$  
\end{lemma}

\begin{proof}
The link over a cone with base $\mathbb{P}_{\R}^1\times \mathbb{P}_{\R}^2$ is the total space of the Hopf fibration over the
base  (see definition \ref{hopf}), hence it is diffeomorphic to the double covering of $\mathbb{P}_{\R}^1\times \mathbb{P}_{\R}^2.$ We show that such a space $M$ is diffeomorphic to $\mathbb{P}_{\R}^1\times \mathbb{P}_{\R}^2$ or $\mathbb{P}_{\R}^1\times \mathbb{S}^2,$  or to the twisted product $\tilde{\mathbb{S}^1\times \mathbb{S}^2}.$ Indeed, let $h : M \to \mathbb{P}_{\R}^1\times \mathbb{P}_{\R}^2$ be the double covering. Hence $h_*(\pi_1(M))$ has index two in $\pi_1(\mathbb{P}_{\R}^1\times \mathbb{P}_{\R}^2)=\mathbb{Z}\times \mathbb{Z}/2.$ Hence either  $h_*(\pi_1(M))=\mathbb{Z}\times \{0\}$, or $h_*(\pi_1(M))=(2)\times \mathbb{Z}/2$, or  $h_*(\pi_1(M))$ is the subgroup $H$ generated by $(1,1).$ The first case corresponds to 
$\mathbb{P}_{\R}^1\times S^2$.  The second case corresponds to $\mathbb{P}_{\R}^1\times \mathbb{P}_{\R}^2.$  The third case corresponds to the twisted product $\tilde{\mathbb{S}^1\times \mathbb{S}^2}$.
\end{proof}

Since  the  link $L_{k,1}$ is a sphere bundle over $\mathbb{P}_{\R}^1,$ 
we have an exact sequence $$ 0=\pi_1(\mathbb{S}^2)\to \pi_1(L_{k,1}) \to \pi_1(\mathbb{P}_{\R}^1)\to 0, $$ 
hence $\pi_1(L_{k,1})=\pi_1(\mathbb{P}_{\R}^1)=\mathbb{Z}.$

Thus the  link $L_{1,1}$ has to be diffeomorphic either to the twisted product $\tilde{\mathbb{S}^1\times \mathbb{S}^2}$ or to $\mathbb{P}_{\R}^1\times \mathbb{S}^2.$ 
Using the theory of Seifert manifolds, we exclude the second possibility. Indeed, the following lemma is true (see \cite[Lemma 2.3.10]{seifert}):

\begin{lemma}\label{lemma1}
If $M$ is an orientable, Seifert fibered space with orbit surface $\mathbb{P}_{\R}^2$ and less than two exceptional fibers, then $M$ is homeomorphic either to a lens space $L(4n, 2n-1)$, or to a Seifert space with orbit space $\mathbb{S}^2$ and three 
exceptional fibers with two of them of index two, or to a connected sum of two copies of $\mathbb{P}_{\R}^3.$ All the relevant fundamental groups are finite except $\pi_1(\mathbb{P}_{\R}^3\#\mathbb{P}_{\R}^3)=\mathbb{Z}/2*\mathbb{Z}/2$.
\end{lemma}

In particular, we see that the space $\mathbb{P}_{\R}^1\times \mathbb{S}^2$ with fundamental group $\mathbb{Z}$ cannot be the total space of a circle bundle over $\mathbb{P}_{\R}^2.$ Thus $L_{1,1}$  is  diffeomorphic to $\tilde{S^1\times \mathbb{S}^2}$. 

Now consider the link $L_{1,2}.$ Since it is a circle bundle over $\mathbb{P}_{\R}^2$ it can be diffeomorphic either to $\mathbb{P}_{\R}^1\times \mathbb{P}_{\R}^2$ or to the twisted product $\tilde{\mathbb{S}^1\times \mathbb{S}^2}.$ If the second possibility holds then we can lift a natural analytic isomorphism 
$\rho: W_{1,1}\to W_{1,2}$ to an analytic isomorphism $\rho': L_{1,1}\to L_{1,2}$ that preserves the Hopf fibration. Indeed, we have the following diagram:

\begin{center}
\begin{picture}(240,160)(-40,40)
\put(0,180){\makebox(0,0)[tl]{$L_{1,1}$}}
\put(178,180){\makebox(0,0)[tl]{$L_{1,2}$}}
\put(0,40){\makebox(0,0)[tl]{$W_{1,1}$}}
\put(170,40){\makebox(0,0)[tl]{$W_{1,2}$}}
\put(80,50){\makebox(0,0)[tl]{$\rho$}}
\put(85,190){\makebox(0,0)[tl]{$\rho' $}}
\put(15,110){\makebox(0,0)[tl]{$\pi$}}
\put(185,110){\makebox(0,0)[tl]{$\pi'$}}
\multiput(25,175)(8,0){17}{\line(1,0){5}}
\put(157,175){\vector(1,0){10}} \put(35,35){\vector(1,0){130}}
\multiput(180,165)(0,-8){14}{\line(0,-1){5}}
\put(180,53){\vector(0,-1){10}}
\multiput(10,165)(0,-8){14}{\line(0,-1){5}}
\put(10,53){\vector(0,-1){10}}
\end{picture}
\end{center}

\vspace{5mm}

\noindent where $\pi,\pi'$ are suitable Hopf fibrations, i.e., they are double coverings.
Since $(\rho\circ \pi)_*(\pi_1(L_{1,1})=\pi'_*(\pi_1(L_{1,2}))=H$, (recall that $H\subset\Bbb Z\times\Bbb Z/(2)$ is a subgroup generated by $(1,1)$),   by \cite[Theorem 6.1]{greenberg} we can lift a mapping $\rho$ to an analytic isomorphism $\rho'.$ By the construction, this lift is an $a$-invariant, i.e., $\rho'(-x)=-\rho'(x).$

By Proposition \ref{alex}, this means  that there is an $a$-invariant subanalytic bi-Lipschitz mapping from $X_{1,1}$ to $X_{1,2}.$
But this mapping has an $a$-invariant graph, and by Corollary \ref{bilipszic}, we have $\deg X_{1,1}=\deg X_{1,2} \mod 2$, a contradiction. Hence $L_{1,2}=\mathbb{P}_{\R}^1\times \mathbb{P}_{\R}^2.$

Let us note that the proof above can be stated in the following more general form
(note that diffeomorphic real algebraic manifolds are Nash isomorphic, see \cite[Corollary 11]{kol}):

\begin{lemma}\label{lemat}
Let $X\subset \mathbb{P}_{\R}^n, Y \subset \mathbb{P}_{\R}^m$ be two smooth projective varieties that are diffeomorphic.
Assume that cones $C_{\R}(X), C_{\R}(Y)$ have connected, homeomorphic links $L,L'.$ Let $G=\pi_1(X), H=\pi_1(L).$ If there is only one  subgroup of $G$ of index two, which is isomorphic to $H$ and induces the same covering as $H$ (i.e., this covering is homeomorphic to $L$), then
$\deg C_{\R}(X)=\deg C_{\R}(Y) \mod 2.$ In particular, this is true if $G$ has only one subgroup of index two.
\end{lemma}

Now consider the link $L_{2,1}.$ Its fundamental group is $\mathbb{Z}$, hence it is diffeomorphic either to the twisted product $\tilde{\mathbb{S}^1\times \mathbb{S}^2}$ or to $\mathbb{P}_{\R}^1\times \mathbb{S}^2.$ By the same argument as above the first possibility is excluded. Hence 
$L_{2,1}=\mathbb{P}_{\R}^1\times \mathbb{S}^2.$ 

\vspace{3mm}

To finish our proof, we need the following:

\begin{lemma}\label{lemma2}
Let $C_{\R}(X)\subset \R^n$ be an algebraic cone of dimension $d>1$ with connected base $X.$ If $\deg C_{\R}(X)$ is odd, then the link of $C_{\R}(X)$ is connected.
\end{lemma}

\begin{proof}
We can assume $d<n.$  Assume that the link $L=A\cup B$ has two connected components. Then $A,B$ are Euler cycles, in particular, they are homological cycles$\mod 2.$ Let $\phi: \mathbb{S}^{n-1} \ni x\mapsto [x]\in \mathbb{P}_{\R}^{n-1}.$ The mapping $\phi$ is continuous and restricted to $A$ is a homeomorphism onto $X.$  Hence $\phi_*([A])=[X].$ Since the cycle $A$ is zero in $H_{d-1}(\mathbb{S}^{n-1},\mathbb{Z}/(2))$, we have $[X]=\phi_*([A])=0$ in $H_{d-1} (\mathbb{P}_{\R}^{n-1},\mathbb{Z}/(2)).$ But $\deg C_{\R}(X)=\deg X \mod 2$- a contradiction.
\end{proof} 

 Let $X$ be as in Theorem \ref{thm:char_P1xP2} and assume deg $C_{\R}(X)$ is odd. If the link $L$ of $C_{\R}(X)$ is not diffeomorphic  to the twisted product $\tilde{\mathbb{S}^1\times \mathbb{S}^2}$, then either $L\cong L_{1,2}$ or $L\cong L_{2,1}$. Since the degrees of the cones $X_{1,2}$ and $X_{2,1}$ are even, by Lemma \ref{lemat} we get a contradiction. In the same way, we can prove that if deg $C_{\R}(X)$ is even, then the link $L$ cannot be diffeomorphic to $L_{1,1}\cong\tilde{\mathbb{S}^1\times \mathbb{S}^2}.$
\end{proof}

\begin{theorem}\label{main_thm}
There exist three semi-algebraically and bi-Lipschitz equivalent algebraic cones $C_{\R}(X), C_{\R}(Y), C_{\R}(B)\subset \R^8$ with non-homeomorphic smooth algebraic bases. In fact, $X\cong \mathbb{P}_{\R}^1\times \mathbb{P}_{\R}^2$,  $Y\cong \mathbb{P}_{\R}^1\times \mathbb{S}^2$, and
$B\cong\tilde{\mathbb{S}^1\times \mathbb{S}^2}.$ In particular, the real version of General Metric Question B has a negative answer.

\end{theorem}

\begin{proof}
Consider the standard embedding $\iota:\mathbb{S}^2 \to \mathbb{P}_{\R}^3$ and let $A=\iota(\mathbb{S}^2).$  Let
$Y:=\phi_1(\mathbb{P}_{\R}^1\times A)$, where $\phi_1: \mathbb{P}_{\R}^1\times \mathbb{P}_{\R}^3\to \mathbb{P}_{\R}^7$ is the Segre embedding. Note that the link of $C_{\R}(Y)$ is connected as a circle bundle over the sphere (we use here the same trick
as in the proof of Theorem \ref{thm:char_P1xP2}).
Since the fundamental group of $Y$ is cyclic, we see that the link of $C_{\R}(Y)$ is $\mathbb{P}_{\R}^1\times \mathbb{S}^2.$ 

Now consider the embedding  $\iota: \tilde{\mathbb{S}^1\times \mathbb{S}^2}\ni [x,y]\mapsto [x:y]\in \mathbb{P}_{\R}^4.$ Let $B=\iota(\tilde{S^1\times \mathbb{S}^2}).$ Hence, $B$ is a hyperquadric in $\mathbb{P}_{\R}^4$ given by the equation $x_1^2+x_2^2=y_1^2+y_2^2+y_3^2.$ Since $C_{\R}(B)$ is a hypersurface its link is a smooth hypersurface on the sphere. Hence, it has an orientable link, see \cite{hs}. However, the space $\tilde{\mathbb{S}^1\times \mathbb{S}^2}$ is non-orientable, this implies that the link of $C_{\R}(B)$ has to be connected. 
Since the group $\pi_1(B)$ is cyclic, the space $B$ has only one connected covering of degree two. Hence, we see that the link is diffeomorphic to $\mathbb{S}^1\times \mathbb{S}^2$. 

Another way to see this (suggested by the referee) is the following:
 the link is a subset of $\mathbb{R}^5$ given by equations $ x_1^2+ x_2^2=y_1^2+y_2^2+y_3^2=\frac{1}{2}$, so it is diffeomorphic to $\mathbb{S}^1\times \mathbb{S}^2.$

Thus, $C_{\R}(B), C_{\R}(X)$ and $C_{\R}(Y)$ have diffeomorphic links.
By \cite[Corollary 11]{kol}, all these links are Nash diffeomorphic. In particular, they are bi-Lipschitz and semi-algebraically equivalent. We conclude the proof  by Proposition \ref{alex}.
\end{proof}

\begin{corollary}
There exist four-dimensional algebraic cones $C_{\R}(X),C_{\R}(Y)\subset \R^8$ such that the following holds:

(a) there exists a semi-algebraic bi-Lipschitz homeomorphism $\phi: C_{\R}(X)\to C_{\R}(Y)$ that transforms every ray $Ox$ into the ray $O\phi(x)$ isometrically, 

(b)  there is no  homeomorphism $C_{\R}(X)\to C_{\R}(Y)$ that transforms
every line passing through the origin onto a line passing through the origin.
\end{corollary}

\begin{proof}
Let $C_{\R}(X), C_{\R}(Y)$ be as in Theorem \ref{main_thm}.  By Proposition \ref{alex} (or rather its proof), there exists a semi-algebraic bi-Lipschitz homeomorphism $\phi: C_{\R}(X)\to C_{\R}(Y)$ that transforms every ray $Ox$ into the ray $O\phi(x)$ isometrically. 

Otherwise, if there were a homeomorphism $C_{\R}(X)\to C_{\R}(Y)$ that transforms
every line passing through the origin onto a line passing through the origin, then it would induce a homeomorphism $X\to Y$, a contradiction.
\end{proof}

\begin{theorem}\label{ost}
(1) The manifolds  $\mathbb{S}^1\times \mathbb{S}^2$ and $\tilde{\mathbb{S}^1\times \mathbb{S}^2}$ are not diffeomorphic to projective  varieties of odd degree. 

(2)  Let $\iota: \mathbb{P}_{\R}^n\to \mathbb{P}_{\R}^N$ be an algebraic embedding. If 
$\deg\iota(\mathbb{P}_{\R}^n)$ is odd, then the link of $C_{\R}(\iota(\mathbb{P}_{\R}^n))$ is $\mathbb{S}^n$, while if $\deg \iota(\mathbb{P}_{\R}^n)$  is even, then the link of $C_{\R}(\iota(\mathbb{P}_{\R}^n))$ is disconnected, i.e., it is homeomorphic to  $\mathbb{P}_{\R}^n
\sqcup \mathbb{P}_{\R}^n$.

(3) A  simply connected real projective variety of positive dimension cannot have odd degree.
\end{theorem}

\begin{proof}
(1) Note that $\tilde{\mathbb{S}^1\times \mathbb{S}^2}$ has an embedding into $\mathbb{P}_{\R}^4$ as a hyperquadric $B$, and the link $L$ of the cone $C_{\R}(B)$ is connected (see the proof of Theorem \ref{main_thm}). 

Now assume that  $\tilde{\mathbb{S}^1\times \mathbb{S}^2}$ has an algebraic embedding $B'$  into some $\mathbb{P}_{\R}^N$  with odd degree. Then its link $L'$ is connected by Lemma \ref{lemma2}, and since the fundamental group of $\tilde{\mathbb{S}^1\times \mathbb{S}^2}$ is cyclic, it has only one connected covering of degree two. Consequently, $L'$ is diffeomorphic to $\mathbb{S}^1\times \mathbb{S}^2.$ Thus, $L$ is diffeomorphic to 
$L'.$ By assumption, the manifolds $B$ and $B'$ are diffeomorphic. Thus, by \cite[Corollary 11]{kol}, they are Nash diffeomorphic, and in particular analytically equivalent. Since the fundamental group of $\tilde{\mathbb{S}^1\times \mathbb{S}^2}$  has only one subgroup of index two, we have by Lemma \ref{lemat} $\deg C_{\R}(B)=\deg C_{\R}(B') \mod 2$, a contradiction.

In a similar way, we can prove that $\mathbb{S}^1\times \mathbb{S}^2$ has no embedding into a projective space with odd degree (note that the manifold $\mathbb{S}^1\times \mathbb{S}^2=\mathbb{P}_{\R}^1\times \mathbb{S}^2$ has a projective embedding of even degree with a connected link of the cone).

(2) Note that $\mathbb{P}_{\R}^n$ has a trivial  embedding into a projective space with  degree one. Let $X$ be another algebraic variety equivalent to  $\mathbb{P}_{\R}^n$ with a connected link. Since $\pi_1(\mathbb{P}_{\R}^n)$ is cyclic by Lemma \ref{lemat} we have 
$1 = \deg  C_{\R}(\mathbb{P}_{\R}^n) = \deg C_{\R}(X) \mod 2.$

(3) Let $X\subset \mathbb{P}^n_{\R}$ be a simply connected algebraic set. If it has an odd degree then by Lemma \ref{lemma2} the cone $C_{\R}(X)$  has connected link $L.$ This means that the Hopf fibration $L\to X$ is a non-trivial covering of degree two. Since $X$ is simply connected, this is a contradiction.  
\end{proof}

\begin{corollary}\label{conj}
For algebraic cones  over $\mathbb{P}_{\R}^1\times\mathbb{P}_{\R}^2$ and algebraic cones over
$\mathbb{P}_{\R}^k$, the Conjecture A$_{\R}$(Lip) has a positive answer.
\end{corollary}

\begin{proof}
It follows from Theorem \ref{thm:char_P1xP2} and Theorem \ref{ost}(2). Indeed, if cones are bi-Lipschitz equivalent, then their links are at least homotopy equivalent. Since it is easy to see that all possible different links, i.e, $\mathbb{P}_{\R}^1\times \mathbb{P}_{\R}^2, 
\mathbb{P}_{\R}^1\times \mathbb{S}^2,\tilde{\mathbb{S}^1\times \mathbb{S}^2}$ and $\mathbb{S}^k, \mathbb{P}_{\R}^k
\sqcup \mathbb{P}_{\R}^k$ are pairwise not homotopy equivalent, the result follows.    \end{proof}

\end{document}